\documentclass{amsart}[12pt]
\usepackage{mathrsfs, amsmath,amssymb}
\usepackage{fullpage}

\newtheorem{pro}{Proposition}
\newtheorem{lem}{Lemma}
\newtheorem{cor}{Corollary}
\newtheorem{rem}{Remark}

\title{Reciprocal of the First hitting time of the boundary of dihedral wedges by a radial Dunkl process}
\author[N. Demni]{Nizar Demni}
\address{IRMAR, Universit\'e de Rennes 1\\ Campus de
Beaulieu\\ 35042 Rennes cedex\\ France}
\email{nizar.demni@univ-rennes1.fr}
\date{\today}
 
\begin{document}
\maketitle
\begin{abstract}
In this paper, we establish an integral representation for the density of the reciprocal of the first hitting time of the boundary of even dihedral wedges by a radial Dunkl process having equal multiplicity values. Doing so provides another proof and extends to all even dihedral groups the main result proved in \cite{Demni1}. We also express the weighted Laplace transform of this density through the fourth Lauricella function and establish similar results for odd dihedral wedges. 
\end{abstract}

\section{Introduction}
The radial Dunkl process associated with a finite reflection group is a diffusion valued in a Weyl chamber and depends on a set of non negative parameters called multiplicity values (\cite{CDGRVY}, Chapters II and III). Its infinitesimal generator is the differential part of the so-called Dunkl-Laplace operator subject to normal reflections at the boundary of the chamber. In the rank-one case, the radial Dunkl process is nothing else but a Bessel process and if all the multiplicities values vanish, then it reduces to the reflected Brownian motion in the Weyl chamber. Another quite interesting instance of the radial Dunkl process is the Brownian motion conditioned to stay in the interior of the Weyl chamber, due to its close connections to the representation theory of semi-simple finite-dimensional Lie algebras (\cite{BBO1}, \cite{BBO2}). 

For small multiplicity values, the radial Dunkl process hits almost surely the boundary of the Weyl chamber. In this case, the tail distribution of the first hitting time of the boundary can be computed from the absolute-continuity relations existing between the process distributions corresponding to different sets of multiplicity values (see \cite{CDGRVY}, p.208). In particular, this probability is expressed, for the infinite families of irreducible root systems $A,B,C,D$, through multivariate confluent hypergeometric function (\cite{Demni}). For dihedral root systems, the tail distribution is given by a series of Jacobi polynomials and one-variable confluent hypergeometric functions, which generalize the expansions derived in \cite{Deb} for the planar Brownian motion in dihedral wedges (in \cite{Demni0}, the formulas were only given for even dihedral groups and similar ones may be derived for odd dihedral groups). Note in passing that combinatorial techniques were used in \cite{Dou-Oco} where the tail distribution of the first exit time from Weyl chambers by a Brownian motion was represented through pfaffians. 

Back to the radial Dunkl process in dihedral wedges, an integral representation of the density of the reciprocal of the first hitting time of the boundary was obtained in \cite{Demni1} for the special angle $\pi/4$ and equal multiplicity values. The obtained integral has the merit to be obviously positive in contrary to its original series representation, and allows for proving an analogue of Dufresne's result on the first hitting time of zero by a Bessel process (see e.g. \cite{MY}) and others due to Vakeroudis and Yor on the exit time from a cone by a planar Brownian motion (\cite{VR}, \cite{Vak-Yor}). It is also more suited to derive asymptotics of this density yet we are intended to investigate them here. In this paper, we shall rather extend the aforementioned integral representation to all even dihedral groups with equal multiplicity values. Our proof of this extension is different from the one written in \cite{Demni1} and appeal to an identity proved in \cite{Del-Dem} and satisfied by ultraspherical polynomials (see eq. \eqref{IdGeg} below). Compared to the integral obtained for the wedge of angle $\pi/4$, the domain of integration is  the multidimensional standard simplex and its integrand still involves a one-variable confluent hypergeometric function yet with a more complicated argument. When the radial Dunkl process starts at a point lying on the bisector of the dihedral wedge, we express the weighted Laplace transform of the density through the fourth Lauricella function, thereby extending to all even dihedral groups Corollary 3.5 proved in \cite{Demni1}. As to the extension Corollary 3.3 in \cite{Demni1}, it is straightforward and we will only allude to it without details in a aside remark. Using an analytic analogy between the infinitesimal generators of radial Dunkl processes corresponding on one side to odd dihedral groups and on the other side to even dihedral groups with equal multiplicity values, we derive an analogous integral representation for odd dihedral wedges where this time, the integrand involves rather a Fox-Wright confluent hypergeometric function. The weighted Laplace transform is then computed but is not expressed through a multivariable hypergeometric series. Let us finally stress that the choice of equal multiplicity values is not a restriction and is rather made for sake of simplicity since the formulas are already cumbersome. Nonetheless, we hope dealing with the case of different multiplicity values in a future paper.  

The paper is organized as follows. The next section recalls some facts on reflection groups and systems, as well as on the radial Dunkl process in dihedral wedges. In the third section, we recall the definitions and some properties of various special functions occurring in the paper. The fourth section is devoted to the derivation of the integral representation of the density for even dihedral wedges and equal multiplicity values, and of its weighted Laplace transform. In the last section, we derive analogous results for odd dihedral wedges.   

\section{Reminder: Dunkl processes in dihedral wedges}
\subsection{Dihedral systems and groups} 
For general facts on root systems and reflection groups, we refer the reader to the monographs \cite{Dun-Xu} and \cite{Hum}. The dihedral group, $\mathcal{D}_2(n), n \geq 3$, consists of orthogonal transformations leaving invariant a regular $n$-sided polygon centered at the origin. Actually, it contains $n$ rotations: 
\begin{equation*}
r_j: z \mapsto ze^{2i\pi j/n}, j =1, \quad \ldots, n, 
\end{equation*}
and $n$ reflections: 
\begin{equation*}
s_j: z \mapsto \overline{z}e^{2i\pi j/n}, \quad j =1, \ldots, n,
\end{equation*}
where $z \in \mathbb{C}$ and $i = \sqrt{-1}$ is the complex imaginary unit. As a finite reflection group, $\mathcal{D}_2(n)$ corresponds to the dihedral root system $I_2(n)$, which may be defined as the set of planar vectors: 
\begin{equation*}
I_2(n) := \{\pm i e^{i\pi l/n},\, 1 \leq l \leq n\}. 
\end{equation*}  
To the following choice of positive root system
\begin{equation*}
\{- i e^{i\pi l/n},\, 1 \leq l \leq n\}, 
\end{equation*}  
corresponds a cone, named the positive Weyl chamber $C$, and is the wedge of dihedral angle $\pi/n$ parametrized in polar coordinates as:
\begin{equation*}
C := \{(r, \theta), \, r > 0, \, 0 < \theta < \pi/n\}.  
\end{equation*}
On the other hand, the dihedral group acts on its root systems in a natural way and this action gives rise to two orbits when $n := 2p, p\geq 2,$ (the roots forming the diagonals and those forming the lines joining the midpoints of the polygon) while it is transitive otherwise. Accordingly, we shall assign to the two orbits non negative real parameters $k_0,k_1,$ respectively when $n$ is even and only one non negative real parameter $k$ otherwise (these parameters are known in the realm of Lie groups as multiplicity values). With these data, we are now ready to recall the definition of the radial Dunkl process associated with dihedral groups (see \cite{Demni0} for further details).  

\subsection{Radial Dunkl processes in dihedral wedges}
The radial Dunkl process $X$ associated with a given reduced root system is a $\overline{C}$-valued diffusion reflected at the boundary $\partial C$ (see \cite{CDGRVY} for further details). In particular, for odd dihedral groups $\mathcal{D}_2(n), n \geq 3$, its infinitesimal generator acts on smooth functions supported in $\overline{C}$ as (\cite{Demni0}):
\begin{equation}\label{G1}
\frac{1}{2}\left[\partial_r^2  + \frac{2nk +1}{r}\partial_r\right] + \frac{1}{r^2}\left[\frac{\partial_{\theta}^2}{2} +nk\cot(n\theta)\partial_{\theta}\right],
\end{equation}
subject to Neumann boundary conditions. While for even dihedral groups $\mathcal{D}_2(2p)$, the action of the infinitesimal generator of $X$ is given by:
\begin{equation}\label{G2}
\frac{1}{2}\left[\partial_r^2  + \frac{2p(k_0+k_1) +1}{r}\partial_r\right]  + \frac{1}{r^2}\left[\frac{\partial_{\theta}^2}{2} +p(k_0\cot(p\theta) - k_1\tan(p\theta))\partial_{\theta}\right].
\end{equation} 
These polar decompositions show that the Euclidean norm of $X$ is a Bessel process and that the cosine of its angular part is, up to a time change, a real Jacobi process on the interval $[0,\pi/n]$ (\cite{Rev-Yor}, \cite{Demni0}). This last observation was the key ingredient in the study  of the first hitting time $T_0$ of $\partial C$ by $X$: 
\begin{equation*}
T_0 := \inf\{t \geq 0, \, X_t \in \partial C\}. 
\end{equation*}
In this respect, recall from \cite{Chy}, Proposition 6 and from \cite{Dem}, Proposition 1, that $T_0$ is almost surely finite if and only if (at least) one of the multiplicity values is strictly less than $1/2$. In particular, let $n = 2p, p \geq 2, x = \rho e^{i\phi} \in C$ and assume $(k_0, k_1) \in (1/2, 1]$. Then, the radial Dunkl process $X$ starting at $X_0 = x$ and associated with the multiplicity values $(1-k_0, 1-k_1)$ hits almost surely the boundary of the dihedral wedge of angle $\pi/(2p)$. As such, the corresponding reciprocal of $T_0$
\begin{equation*}
V_0 := \frac{\rho^2}{2T_0} 
\end{equation*}
is a almost surely positive random variable. A similar statement holds true for odd dihedral wedges. 

\section{Special functions}
In this section, we record the definitions of various special functions occurring in the remainder of the paper as well as some of their properties we will need in our subsequent computations. The reader is referred for instance to \cite{AAR}, \cite{Er}, \cite{Erd}, \cite{Man-Sri}. 
We start with the Gamma integral: 
\begin{equation*}
\Gamma(z) = \int_0^{\infty} e^{-u}u^{z-1} du, \quad \Re(z) > 0, 
\end{equation*}
and the Legendre duplication formula: 
\begin{equation}\label{Leg}
\sqrt{\pi}\Gamma(2z+1) = 2^{2z-1}\Gamma\left(z+\frac{1}{2}\right)\Gamma(z+1).
\end{equation}
Then, we recall the Pochhammer symbol:
\begin{equation*}
(a)_k = (a+k-1)\dots(a+1)a, \quad a \in \mathbb{C}, \, k \in \mathbb{N}, 
\end{equation*}
which may be written when $a > 0$ as:
\begin{equation*}
(a)_k = \frac{\Gamma(a+k)}{\Gamma(a)}.
\end{equation*} 
Next comes the generalized hypergeometric function defined by the series: 
\begin{equation*}
{}_rF_q((a_i, 1 \leq i \leq r), (c_j, 1 \leq j \leq q); z) = \sum_{m \geq 0}\frac{\prod_{i=1}^r(a_i)_m}{\prod_{j=1}^q(c_j)_m}\frac{z^m}{m!},
\end{equation*}
provided it converges absolutely. Here, an empty product equals one and the parameters $(a_i, 1 \leq i \leq r)$ are complex numbers while $(c_j, 1 \leq j \leq q) \in \mathbb{C} \setminus -\mathbb{N}$. If $a_i = -n \in - \mathbb{N}$ for some $1 \leq i \leq r$,
then the hypergeometric series terminates and as such, reduces  to a polynomial of degree $n$. For instance, the $j$-th Gegenbauer polynomial of parameter $\lambda$ is defined through the Gauss hypergeometric function ${}_2F_1$:
\begin{equation}\label{GegJac}
C_j^{(\lambda)}(z) := \frac{(2\lambda)_j}{j!}{}_2F_1\left(-j, j+2\lambda, \lambda+\frac{1}{2}, \frac{1-z}{2}\right), \quad \lambda \neq 0.
\end{equation}
For this polynomial and this hypergeometric function, the following argument transformations hold: 
\begin{equation}\label{Sym}
C_j^{(\lambda)}(-z) = (-1)^j C_j^{(\lambda)}(z),  
\end{equation}
\begin{equation}\label{ET}
{}_2F_1(a,c_1, c_2; z) = (1-z)^{-a} {}_2F_1\left(a, c_2-c_1, c_2; \frac{z}{z-1}\right), \quad |\arg(1-z)| < \pi.
\end{equation}
Now, the confluent hypergeometric series ${}_1F_1$ converges in the whole complex plane: 
\begin{equation}\label{ExpConf}
{}_1F_1(a,c, z) = \sum_{j \geq 0} \frac{(a)_j}{(c)_j} \frac{z^j}{j!}, 
\end{equation}
and admits the following integral representation valid for $\Re(c) > \Re(a) > 0$: 
\begin{align}
{}_1F_1\left(a, c, z\right) = \frac{\Gamma(c)}{\Gamma(a)\Gamma(c-a)}\int_0^1e^{uz}u^{a-1}(1-u)^{c-a-1}du. \label{Euler}
\end{align}
We shall also need the fourth Lauricella function $F_D^{(p-1)}$ in $p-1, p \geq 2$ variables: 
\begin{equation}\label{Laur}
F_D^{(p-1)}(a, d_1, \ldots, d_{p-1}, d_p; z_1, \ldots, z_{p-1}) := \sum_{m_1, \ldots, m_{p-1} \geq 0} \frac{(a)_{m_1+\ldots+m_{p-1}}}{(d_p)_{m_1+\ldots+m_{p-1}}}\prod_{s=1}^{p-1} (d_s)_{m_s} z_s^{m_s},
\end{equation}
which converges for $|z_s| < 1, 1 \leq s \leq p-1$, where $a, d_1, \ldots, d_{p-1} \in \mathbb{C}$ and $d_p \in \mathbb{C} \setminus \mathbb{N}$. When $\Re(d_s) > 0, 1 \leq s \leq p,$ it admits the following Euler-type integral representation 
(see e.g. \cite{Lij-Reg}, eq. (2.1)): 
\begin{multline}\label{IRD}
F_D^{(p-1)}(a, d_1, \ldots, d_{p-1}, d_1 + \ldots + d_{p-1} + d_p; z_1, \ldots, z_{p-1}) = \frac{\Gamma(a+d_1+\ldots+d_{p-1})}{\Gamma(a)\prod_{s=1}^{p-1}\Gamma(d_s)}\int_{\Sigma_p} du_1 \ldots du_{p-1}  
\\ \left(1-u_1z_1-\ldots - u_{p-1}z_{p-1}\right)^{-a} \prod_{s=1}^{p}u_s^{d_s-1},
\end{multline}
where 
\begin{equation*}
\Sigma_p := \{(u_1, \ldots, u_p), \, u_1, \ldots, u_p \geq 0, \, u_1 + \ldots + u_p = 1\}
\end{equation*}
is the standard simplex in $\mathbb{R}^p$. Finally, we recall the Dirichlet integral: for any $\beta_1, \ldots, \beta_p > 0$,
\begin{equation}\label{Dir}
\int_{\Sigma_p} du_1 \ldots du_{p-1}\prod_{s=1}^p u_s^{\beta_s-1} = \frac{\Gamma(\beta_1) \ldots \Gamma(\beta_p)}{\Gamma(\beta_1+\ldots+\beta_p)},
\end{equation}
and we denote
\begin{equation*}
\mu^{s}(du) = \frac{\Gamma(s+1/2)}{\sqrt{\pi}\Gamma(s)}(1-u^2)^{s-1} {\bf 1}_{[-1,1]}(u)du, \quad s > 0,
\end{equation*}
the symmetric Beta distribution. 

\section{Integral representation of the density of $V_0$: even dihedral wedges and equal multiplicity values}
Let $k_0 = k_1 \equiv k \in (1/2,1]$ and consider the radial Dunkl process in an even dihedral wedge of angle $\pi/(2p), p \geq 2,$ and associated with $(1-k, 1-k)$. In \cite{Demni1}, it was shown that\footnote{The result is also valid for $p=1$} the density of $V_0$ is, up to a normalizing constant, the even part of the following series (see eq. (3.2), p.146): 
\begin{equation}\label{Serie1}
\sin^{2\nu}(2p\phi) e^{-v}v^{2p\nu-1}\sum_{j \geq 0}\frac{\Gamma(p(j+1))}{\Gamma(2p(j+k))}v^{pj}{}_1F_1\left(p(j+1), 2p(j+k)+1, v\right)C_{j}^{(k)}(\cos(2p\phi)),
\end{equation}
where
\begin{equation*}
v > 0, \quad \nu := k - \frac{1}{2}.
\end{equation*}
For the particular value $p =2$, the following integral representation of \eqref{Serie1} was obtained in \cite{Demni1}, Lemma 3.2: 
\begin{multline}\label{Formula}
\sum_{j \geq 0}\frac{\Gamma(2(j+1))}{\Gamma(4(j+k))}v^{2j}{}_1F_1\left(2(j+1), 4(j+k) + 1, v\right)C_{j}^{(k)}(\cos(4\phi)) = \frac{1}{\Gamma(4k)} \\ \int {}_1F_1\left(2, 2\nu+\frac{3}{2}; \frac{v(1-\cos(2\phi)u)}{2}\right)\mu^{k}(du).
\end{multline}
The extension of \eqref{Formula} to all values $p \geq 2$ is as follows: 

\begin{pro}
For any $p \geq 2$, 
\begin{multline}\label{Formula1}
\sum_{j \geq 0}\frac{\Gamma(p(j+1))}{\Gamma(2p(j+k))}v^{pj}{}_1F_1\left(p(j+1), 2p(j+k)+1, v\right)C_{j}^{(k)}(\cos(2p\phi))  = \frac{\Gamma(p)\Gamma(pk)}{[\Gamma(k)]^p\Gamma(2kp)} \int_{\Sigma_p}du_1 \ldots du_{p-1} \prod_{s=1}^p u_s^{k-1} 
\\ {}_1F_1 \left(p, pk+\frac{1}{2}; v \sum_{s=1}^p u_s \cos^{2}\left(\phi + \frac{s\pi}{p}\right)\right).
\end{multline}
\end{pro}

\begin{proof}
From \eqref{ExpConf}, we readily derive
\begin{align*}
\frac{\Gamma(p(j+1))}{\Gamma(2p(j+k))}{}_1F_1\left(p(j+1), 2p(j+k)+1, v\right) & = 2p(j+k) \sum_{m \geq 0} \frac{\Gamma(a_j+m)}{\Gamma(b_j+m+1)} \frac{v^m}{m!},
\end{align*}
whence
\begin{multline*}
\sum_{j \geq 0}\frac{\Gamma(p(j+1))}{\Gamma(2p(j+k))}v^{pj}{}_1F_1\left(p(j+1), 2p(j+k)+1, v\right)C_{j}^{(k)}(\cos(2p\phi)) =  (2p) \sum_{j, m\geq 0}  (j+k) \\ 
\frac{\Gamma(pj+m + p)}{\Gamma(2pj+m + 2pk +1)m!} v^{pj+m} C_{j}^{(k)}(\cos(2p\phi)) 
\\ = \sum_{N \geq 0}v^N \Gamma(N+p) \sum_{\substack{m,j \geq 0 \\ N =pj+m}} \frac{2p(j+k)}{\Gamma(2pj+m + 2pk +1)m!}  C_{j}^{(k)}(\cos(2p\phi)).
\end{multline*}
Now, recall from \cite{Del-Dem}, Proposition 1,  the following identity: for any integers $M \geq 0, q \geq 1$, any real numbers $k > 0, \xi \in [0,\pi]$, we have:  
\begin{equation}\label{IdGeg}
\sum_{\substack{m,j \geq 0 \\ M = 2m+qj}}\frac{q(j+k)}{m!\Gamma(q(j+k)+m+1)}C_j^{(k)}(\cos \xi) = \frac{2^M}{\Gamma(M+qk)} \sum_{\substack{j_1, \ldots, j_q \geq 0 \\ j_1+\cdots j_q = M}}  
(k)_{j_1}\ldots (k)_{j_{q}} \frac{[b_{1,q}(\xi)]^{j_1}}{j_1!}\cdots  \frac{[b_{q,q}(\xi)]^{j_{q}}}{j_{q}!},
\end{equation}
where 
\begin{equation*}
b_{s,q}(\xi) := \cos\left(\frac{\xi + 2s\pi}{q}\right), \quad s = 1, \ldots, q.
\end{equation*} 
In particular, if $M = 2N, q = 2p, \xi = 2p\phi$, then \eqref{IdGeg} specializes to:
\begin{align*}
\sum_{\substack{m,j \geq 0 \\ N = m+pj}}\frac{2p(j+k)}{m!\Gamma(2p(j+k)+m+1)}C_j^{(k)}(\cos(2p\phi)) = \frac{2^{2N}}{\Gamma(2N+2pk)} \sum_{\substack{j_1, \ldots, j_{2p} \geq 0 \\ j_1+\cdots j_{2p} = 2N}}  
\prod_{s=1}^{2p} \frac{(k)_{j_s}}{j_s!} \cos^{j_s} \left(\phi + \frac{s\pi}{p}\right).
\end{align*}
As a result, 
\begin{multline*}
\sum_{j \geq 0}\frac{\Gamma(p(j+1))}{\Gamma(2p(j+k))}v^{pj}{}_1F_1\left(p(j+1), 2p(j+k)+1, v\right) C_{j}^{(k)}(\cos(2p\phi)) = \sum_{N \geq 0}v^N  \frac{2^{2N}\Gamma(N+p)}{\Gamma(2N+2pk)} 
\\ \sum_{\substack{j_1, \ldots, j_{2p} \geq 0 \\ j_1+\cdots j_{2p} = 2N}}  \prod_{s=1}^{2p} \frac{(k)_{j_s}}{j_s!} \cos^{j_s}\left(\phi + \frac{s\pi}{p}\right)
 = \frac{\Gamma(p)}{\Gamma(2pk)} \sum_{N \geq 0}v^N\frac{(p)_N}{(pk)_N(pk+(1/2))_N} \\ 
  \sum_{\substack{j_1, \ldots, j_{2p} \geq 0 \\ j_1+\cdots j_{2p} = 2N}}  \prod_{s=1}^{2p} \frac{(k)_{j_s}}{j_s!} \cos^{j_s}\left(\phi + \frac{s\pi}{p}\right).
\end{multline*}
Now, since $\cos(a+\pi) = -\cos(a)$, then the finite sum in last equality may be written as: 
\begin{equation}\label{Sum1}
\sum_{\substack{j_1, \ldots, j_{2p} \geq 0 \\ j_1+\cdots j_{2p} = 2N}} (-1)^{j_1+\dots+j_p} \prod_{s=1}^{2p} \frac{(k)_{j_s}}{j_s!}  \prod_{s=1}^p \cos^{j_s + j_{s+p}}\left(\phi + \frac{s\pi}{p}\right), 
\end{equation}
which may be further simplified as follows. Consider a $2p$-tuple of integers in the sum \eqref{Sum1} arranged in $p$ pairs: 
\begin{equation*}
(j_1,j_{p+1}), \ldots, (j_p, j_{2p}), \quad \sum_{s=1}^{2p} j_s = 2N, 
\end{equation*}
and assume that one pair consists of an even and an odd integers. Then, the $2p$-tuple obtained from the previous one by permuting the even and the odd integers in the given pair (while keeping the other pairs unchanged) cancels the latter since the expression
\begin{equation*}
\prod_{s=1}^{2p} \frac{(k)_{j_s}}{j_s!}  \prod_{s=1}^p \cos^{j_s + j_{s+p}}\left(\phi + \frac{s\pi}{p}\right) 
\end{equation*}
is the same for both tuples. As a matter of fact, the only tuples that remains in \eqref{Sum1} after cancellations are those for which $j_s+j_{s+p}$ is even, therefore 
\begin{multline}\label{Sum2}
\sum_{\substack{j_1, \ldots, j_{2p} \geq 0 \\ j_1+\cdots j_{2p} = 2N}} (-1)^{j_1+\dots+j_p} \prod_{s=1}^{2p} \frac{(k)_{j_s}}{j_s!}  \prod_{s=1}^p \cos^{j_s + j_{s+p}}\left(\phi + \frac{s\pi}{p}\right) = 
\sum_{\substack{j_1, \ldots, j_{2p} \geq 0 \\ j_1+\cdots j_{2p} = 2N \\ j_s+j_{s+p} \textrm{is even}}}  (-1)^{j_1+\dots+j_p} \\ \prod_{s=1}^{2p} \frac{(k)_{j_s}}{j_s!}  \prod_{s=1}^p \cos^{j_s+j_{s+p}}\left(\phi + \frac{s\pi}{p}\right).
\end{multline}
Setting $m_s = j_s+j_{s+p}$ and using the identity 
\begin{equation*}
\sum_{j_s = 0}^{2m_s} (-1)^{j_s} \frac{(k)_{j_s}(k)_{2m_s-j_s}}{j_s!(2m_s - j_s)!} = \frac{(k)_{m_s}}{m_s!},
\end{equation*}
which is readily checked by equating generating functions of both sides, then the right-hand side of \eqref{Sum2} may be written as 
\begin{align*}
\sum_{\substack{m_1, \ldots, m_{p} \geq 0 \\ m_1+\cdots m_{p} = N}} \prod_{s=1}^p \cos^{2m_s}\left(\phi + \frac{s\pi}{p}\right)  \sum_{j_s = 0}^{2m_s} (-1)^{j_s} \frac{(k)_{j_s}(k)_{2m_s-j_s}}{j_s!(2m_s - j_s)!} = 
\sum_{\substack{m_1, \ldots, m_{p} \geq 0 \\ m_1+\cdots m_{p} = N}} \prod_{s=1}^p \cos^{2m_s}\left(\phi + \frac{s\pi}{p}\right) \frac{(k)_{m_s}}{m_s!}.
\end{align*}

Consequently, 
\begin{multline*}
\sum_{j \geq 0}\frac{\Gamma(p(j+1))}{\Gamma(2p(j+k))}v^{pj}{}_1F_1\left(p(j+1), 2p(j+k)+1, v\right) C_{j}^{(k)}(\cos(2p\phi))  = \frac{\Gamma(p)}{\Gamma(2pk)}\sum_{N \geq 0}v^N\frac{(p)_N}{(pk)_N(pk+(1/2))_N}  \\ 
\sum_{\substack{m_1, \ldots, m_{p} \geq 0 \\ m_1+\cdots m_{p} = N}} \prod_{s=1}^p \cos^{2m_s}\left(\phi + \frac{s\pi}{p}\right) \frac{(k)_{m_s}}{m_s!}
= \frac{\Gamma(p)}{\Gamma(2pk)}\sum_{m_1, \ldots, m_{p} \geq 0 }\frac{(p)_{m_1+\dots+m_p}}{(pk)_{m_1+\dots+m_p}(pk+(1/2))_{m_1+\dots+m_p}} \\ 
\prod_{s=1}^p \frac{(k)_{m_s}}{m_s!} \left\{v\cos \left(\phi + \frac{s\pi}{p}\right)\right\}^{2m_s}.
\end{multline*}
Finally, substituting $\beta_s = k+ m_s$ in \eqref{Dir} and using the multinomial theorem: 
\begin{equation*}
(a_1+\ldots +a_p)^N = \sum_{m_1+\ldots + m_p = N} \frac{N!}{m_1!\ldots m_p!} \prod_{s=1}^p a_s^{m_s},
 \end{equation*}
together with \eqref{ExpConf}, we end up with : 
\begin{multline*}
\frac{\Gamma(p)}{\Gamma(2pk)}\sum_{m_1, \ldots, m_{p} \geq 0 }\frac{(p)_{m_1+\dots+m_p}}{(pk)_{m_1+\dots+m_p}(pk+(1/2))_{m_1+\dots+m_p}}  
\prod_{s=1}^p \frac{(k)_{m_s}}{m_s!} \left\{v\cos^2 \left(\phi + \frac{s\pi}{p}\right)\right\}^{m_s} = \frac{\Gamma(p)\Gamma(pk)}{[\Gamma(k)]^p\Gamma(2pk)} 
\\ \int_{\Sigma_p} du_1\ldots du_{p-1} \sum_{m_1, \ldots, m_{p} \geq 0 }\frac{(p)_{m_1+\dots+m_p}}{(pk+(1/2))_{m_1+\dots+m_p}}  
\prod_{s=1}^{p} \frac{u_s^{k+m_s-1}}{m_s!} \left\{v\cos^2 \left(\phi + \frac{s\pi}{p}\right)\right\}^{m_s}
\\ = \frac{\Gamma(p)\Gamma(pk)}{[\Gamma(k)]^p\Gamma(2pk)} 
 \int_{\Sigma_p} du_1\ldots du_{p-1} \prod_{s=1}^{p} u_s^{k-1}\sum_{N \geq 0}\frac{(p)_N}{(pk+(1/2))_N} \\ 
 \sum_{m_1+\ldots + m_p= N}  \prod_{s=1}^{p} \frac{u_s^{m_s}}{m_s!} \left\{v\cos^2 \left(\phi + \frac{s\pi}{p}\right)\right\}^{m_s} 
 \\ = \frac{\Gamma(p)\Gamma(pk)}{[\Gamma(k)]^p\Gamma(2pk)} 
 \int_{\Sigma_p} du_1\ldots du_{p-1} \prod_{s=1}^p u_s^{k-1}\sum_{N \geq 0}\frac{(p)_N}{(pk+(1/2))_N} \frac{1}{N!}\left\{v\sum_{s=1}^pu_s\cos^2 \left(\phi + \frac{s\pi}{p}\right)\right\}^N. 
\end{multline*}

Using \eqref{ExpConf}, the proposition is proved. 
\end{proof}

\begin{rem}
If $p=2$, then \eqref{Formula1} simplifies to 
\begin{multline*}
\frac{\Gamma(2k)}{[\Gamma(k)]^2\Gamma(4k)} \int_0^1u^{k-1}(1-u)^{k-1}  {}_1F_1 \left(2, 2k+\frac{1}{2}; v \left\{u \sin^{2}\left(\phi \right) + (1-u) \cos^2(\phi)\right\}\right) du = \frac{\Gamma(2k)}{[\Gamma(k)]^2\Gamma(4k)}  
\\ \int_0^1u^{k-1}(1-u)^{k-1}  {}_1F_1 \left(2, 2k+\frac{1}{2}; v \left\{\frac{1+ (1-2u) \cos(2\phi)}{2}\right\}\right) du.
\end{multline*}
Performing the variable change $u \mapsto (1-u)/2$ in this integral and using Legendre duplication formula \eqref{Leg}, one recovers \eqref{Formula} (recall $k = \nu + (1/2))$. 
\end{rem}

Keeping in mind \eqref{Serie1}, we get: 
\begin{cor}\label{Cor1}
Let $k \in (1/2,1]$ and $p \geq 2$. Then, the density of $V_0$ admits the following integral representation (up to a normalizing constant depending only on $p,k$): 
\begin{multline*}
\sin^{2\nu}(2p\phi) e^{-v}v^{2p\nu-1} \int_{\Sigma_p} du_1 \ldots du_{p-1} \prod_{s=1}^p u_s^{k-1}
\left\{ {}_1F_1 \left(p, p\nu+\frac{(p+1)}{2}; v \sum_{s=1}^p u_s \cos^{2}\left(\phi + \frac{s\pi }{p}\right)\right) \right. + \\ \left.{}_1F_1 \left(p, p\nu+\frac{(p+1)}{2}; v \sum_{s=1}^p u_s \cos^{2}\left(\phi - \frac{(2s-1)\pi }{2p}\right)\right)\right\}.
\end{multline*}
\end{cor}
\begin{proof}
As claimed in the beginning of this section, the density of $V_0$ is half of the sum of both series: 
\begin{equation*}
\sum_{j \geq 0}(\pm 1)^j\frac{\Gamma(p(j+1))}{\Gamma(2p(j+k))}v^{pj}{}_1F_1\left(p(j+1), 2p(j+k)+1, v\right)C_{j}^{(k)}(\cos(2p\phi)).
\end{equation*}
Using the symmetry relation \eqref{Sym}, we get: 
\begin{equation*}
(-1)^jC_{j}^{(k)}(\cos(2p\phi)) = C_{j}^{(k)}\left(\cos\left(2p\left(\frac{\pi}{2p} - \phi\right)\right)\right),
\end{equation*}
whence 
\begin{multline}\label{Formula2}
\sum_{j \geq 0}\frac{\Gamma(p(j+1))}{\Gamma(2p(j+k))}v^{pj}{}_1F_1\left(p(j+1), 2p(j+k), v\right)(-1)^jC_{j}^{(k)}(\cos(2p\phi))  = \frac{\Gamma(p)\Gamma(pk)}{[\Gamma(k)]^p\Gamma(2kp)} \int_{\Sigma_p}du_1\ldots du_{p-1} 
\\ \prod_{s=1}^p u_s^{k-1} {}_1F_1 \left(p, pk+\frac{1}{2}; v \sum_{s=1}^p u_s \cos^{2}\left(\phi - \frac{\pi(2s+1)}{2p} \right)\right).
\end{multline}
Since
\begin{equation*}
\cos^{2}\left(\phi - \frac{\pi(2p+1)}{2p} \right) = \cos^{2}\left(\phi  - \frac{\pi}{2p} \right),
\end{equation*}
then equality \eqref{Formula2} is rewritten: 
\begin{multline}\label{Formula3}
\sum_{j \geq 0}\frac{\Gamma(p(j+1))}{\Gamma(2p(j+k))}v^{pj}{}_1F_1\left(p(j+1), 2p(j+k), v\right)(-1)^jC_{j}^{(k)}(\cos(2p\phi))  = \frac{\Gamma(p)\Gamma(pk)}{[\Gamma(k)]^p\Gamma(2kp)} \int_{\Sigma_p}du_1\ldots du_{p-1} 
\\ \prod_{s=1}^p u_s^{k-1} {}_1F_1 \left(p, pk+\frac{1}{2}; v \sum_{s=1}^p u_s \cos^{2}\left(\phi - \frac{\pi(2s-1)}{2p} \right)\right),
\end{multline}
and the corollary is proved. 
\end{proof}

\begin{rem}
With \eqref{Formula} in hands, one readily extends Corollary 3.3 in \cite{Demni1} to all $p \geq 2$ provided 
\begin{equation*}
p(k -1) + \frac{1}{2} > 0, 
\end{equation*}
ensuring the validity of the Euler-type integral representation \eqref{Euler}. As to the extension of Corollary 3.5 in \cite{Demni1}, it goes as follows: 
\end{rem}

\begin{cor}\label{Cor2}
Assume the radial Dunkl process starts at a point lying on the bisector of the Weyl chamber: 
\begin{equation*}
x = \rho e^{i\pi/(4p)}, \quad \rho > 0. 
\end{equation*}
Denote $\mathbb{E}_{\rho, \pi/(4p)}^{(-\nu)}$ the distribution of this process associated with the common multiplicity value $(1-k)$. Then, for any $y > 0$,
\begin{multline*}
\mathbb{E}_{\rho, \pi/(4p)}^{(-\nu)}\left(V^{(p+1)/2- p\nu} e^{-yV_0}\right) \propto \frac{\sin^{2\nu}(2p\phi)}{(1+y)^{p(\nu-1)+(p+1)/2}} \left[y + \sin^2\left(\frac{\pi}{4p}\right)\right]^{-p} 
\\ F_D^{(p-1)}\left(p, \underbrace{k, \ldots, k}_{p-1}, pk; - \frac{\sin[3\pi/(2p)] \sin[(\pi/p)]}{y+\sin^2(\pi/(4p))}, \ldots, -\frac{\sin[(2p+1)\pi/(2p)] \sin[((p-1)\pi/p)]}{y+\sin^2(\pi/(4p))}\right),
\end{multline*}
where the coefficient of proportionality only depends on $k,p$. 
\end{cor}
\begin{proof}
The choice $\phi = \pi/(4p)$ ensures the equality: 
\begin{equation*}
\sum_{s=1}^p u_s \cos^{2}\left(\frac{\pi}{4p} + \frac{s\pi }{p}\right) = \sum_{s=1}^p u_s \cos^2\left( \frac{(4s+1)\pi}{4p} \right)  = \sum_{s=1}^p u_s \cos^{2}\left(\frac{\pi}{4p} - \frac{\pi(2s-1)}{2p}\right),
\end{equation*} 
so that both confluent hypergeometric functions displayed in Corollary \ref{Cor1} coincide. Now, set 
\begin{equation*}
M_p(u_1, \ldots, u_p) := \sum_{s=1}^p u_s \cos^{2}\left(\frac{(4s+1)\pi}{4p}\right). 
\end{equation*}
Then, the Gamma integral together with the generalized binomial Theorem entail:  
\begin{multline*}
\int_0^{\infty} v^{p\nu + (p-1)/2} e^{-v(1+y)} {}_1F_1 \left(p, p\nu+\frac{(p+1)}{2}; vM_p(u_1, \ldots, u_p)\right) = \frac{\Gamma(p\nu+(p+1)/2)}{(1+y)^{p\nu+(p+1)/2}} \sum_{N = 0}^{\infty} \frac{(p)_N}{N!} \\ 
\left[\frac{M_p(u_1, \ldots, u_p)}{(1+y)}\right]^N  = \frac{1}{(1+y)^{p(\nu-1)+(p-3)/2}}\frac{1}{[1+y - M_p(u_1, \ldots, u_p)]^p}.
\end{multline*}
Next, we compute
\begin{align*}
1+y - M_p(u_1, \ldots, u_p) & = 1+y - \sum_{s=1}^{p-1} u_s\cos^{2}\left(\frac{(4s+1)\pi}{4p}\right) - (1-u_1-\ldots - u_{p-1})\cos^2\left(\frac{\pi}{4p}\right)
\\& = y + \sin^2\left(\frac{\pi}{4p}\right) - \sum_{s=1}^{p-1}u_s\left\{\cos^{2}\left(\frac{(4s+1)\pi}{4p}\right) - \cos^2\left(\frac{\pi}{4p}\right)\right\}
\\& = y + \sin^2\left(\frac{\pi}{4p}\right) - \frac{1}{2}\sum_{s=1}^{p-1}u_s\left\{ \cos\left(\frac{(4s+1)\pi}{2p}\right) - \cos\left(\frac{\pi}{2p}\right)\right\}
\\& = y + \sin^2\left(\frac{\pi}{4p}\right) - \sum_{s=1}^{p-1}u_s \sin\left(\frac{(2s+1)\pi}{2p}\right)\sin\left(-\frac{s\pi}{p}\right)
\\& = \left[y + \sin^2\left(\frac{\pi}{4p}\right)\right]^{-1} \left\{1- \sum_{s=1}^{p-1}u_s \frac{\sin[(2s+1)\pi/(2p)] \sin[-(s\pi/p)]}{y+\sin^2(\pi/(4p))}\right\}. 
\end{align*}
Using \eqref{IRD}, we are done. 
\end{proof}
\begin{rem}
If $p=2$, then 
\begin{equation*}
\sin^2\left(\frac{\pi}{8}\right) = \frac{\sqrt{2}-1}{2\sqrt{2}}
\end{equation*}
and the Lauricella function ${}F_D^{(1)}$ coincides with the Gauss hypergeometric function: 
\begin{equation*}
{}_2F_1\left(2, k, 2k; \frac{2}{1- \sqrt{2}(1+2y)}\right).
\end{equation*}
Using the argument transformation \eqref{ET}, this hypergeometric function is transformed into: 
\begin{equation*}
{}_2F_1\left(2, k, 2k; \frac{2}{1- \sqrt{2}(1+2y)}\right) = \frac{(1-\sqrt{2}(1+2y))^2}{( 1+\sqrt{2}(1+2y))^2} {}_2F_1\left(2, k, 2k; \frac{2}{1+ \sqrt{2}(1+2y)}\right).
\end{equation*}
As a result, 
\begin{equation*}
\mathbb{E}_{\rho, \pi/(8)}^{(-\nu)}\left(v^{(3)/2- 2\nu} e^{-yV_0}\right) \propto \frac{\sin^{2\nu}(4\phi)}{(1+y)^{2\nu-(1/2)}}\frac{1}{( 1+\sqrt{2}(1+2y))^2} {}_2F_1\left(2, k, 2k; \frac{2}{1+ \sqrt{2}(1+2y)}\right)
\end{equation*}
which is the expression derived in the proof of Corollary 3.5 in \cite{Demni1}. 
\end{rem}

\section{Integral representation of the density of $V_0$: odd dihedral wedges} 
Our previous reasoning applies to odd dihedral groups. Indeed, notice that if $k_0 = k_1$, then the infinitesimal generator \eqref{G2} reduces to 
\begin{equation*}
\frac{1}{2}\left[\partial_r^2  + \frac{4pk_0 +1}{r}\partial_r\right]  + \frac{1}{r^2}\left[\frac{\partial_{\theta}^2}{2} +2pk_0\cot(2p\theta) \right], \quad \theta \in [0, \pi/(2p)],
\end{equation*}
which is exactly the infinitesimal generator \eqref{G1} after the identifications $k_0 \leftrightarrow k, n \leftrightarrow 2p$. Since the tail distribution of $T_0$ starting at $x$ is the unique solution of the heat equation with appropriate boundary conditions (see e.g. \cite{Deb}), then this probability has the same expression in both cases (i.e., corresponding to odd dihedral wedges on the one side and to even ones with equal multiplicity values on the other side) under the previous identifications. Of course, one can also mimic the derivation of the tail distribution written in \cite{Demni0}, paragraph 7.1 and retrieve the same expression. Hence, if $n \geq 3$ is odd and $k \in (1/2,1]$, then the radial Dunkl process associated with the dihedral group $\mathcal{D}_2(n)$ and corresponding to the multiplicity value $1-k \in [0,1/2)$ hits $\partial C$ almost surely and the density of $V_0$ is (up to a normalizing constant) the even part of the following series:  
\begin{equation}\label{Density0}
\sin^{2\nu}(n\phi) e^{-v}v^{n\nu-1}\sum_{j \geq 0}\frac{\Gamma(n(j+1)/2)}{\Gamma(n(j+k))}(\sqrt{v})^{j}{}_1F_1\left(n(j+1)/2, n(j+k)+1, v\right)C_{j}^{(k)}(\cos(n\phi)), \quad v > 0.
\end{equation}
Using again \eqref{ExpConf} followed by \eqref{IdGeg}, we may rewrite \eqref{Density0} as: 
\begin{multline*}
\sum_{j \geq 0}\frac{\Gamma(n(j+1)/2)}{\Gamma(n(j+k))} (\sqrt{v})^{j} {}_1F_1\left(n(j+1)/2, n(j+k)+1, v\right)C_{j}^{(k)}(\cos(n\phi)) = \sum_{j, m \geq 0} n(j+k) \\ 
\frac{\Gamma(n(j+1)/2+m)}{m!\Gamma(nj+m+nk+1)} (\sqrt{v})^{nj+2m}C_{j}^{(k)}(\cos(n\phi))
= \sum_{N \geq 0} (\sqrt{v)}^N\Gamma((N+n)/2) 
\\ \sum_{N = nj+2m} \frac{n(j+k)}{m!\Gamma(nj+m+nk+1)} C_{j}^{(k)}(\cos(n\phi))
= \sum_{N \geq 0} (2\sqrt{v})^N\frac{\Gamma((N+n)/2))}{\Gamma(N+nk)}  
\\ \sum_{\substack{j_1, \ldots, j_n \geq 0 \\ j_1+\cdots j_n = N}}  (k)_{j_1}\ldots (k)_{j_{n}} \frac{[b_{1,n}(n\phi)]^{j_1}}{j_1!}\cdots  \frac{[b_{n,n}(n\phi)]^{j_{n}}}{j_{n}!}.
\end{multline*}
Finally, the Dirichlet integral \eqref{Dir} and the multinomial Theorem lead again to:
\begin{pro}
Let $n \geq 3$ be an integer and $k > 0$. Then, for any $v > 0$, 
\begin{multline*}
\sum_{j \geq 0}\frac{\Gamma(n(j+1)/2)}{\Gamma(n(j+k))}v^{j/2}{}_1F_1\left(n(j+1)/2, n(j+k)+1, v\right)C_{j}^{(k)}(\cos(n\phi)) = \frac{1}{[\Gamma(k)]^n} \int_{\Sigma_n} du_1\ldots du_{n-1} \prod_{s=1}^n u_s^{k-1} 
\\ \sum_{N \geq 0} \frac{\Gamma((N+n)/2))}{(nk)_N} \frac{(2\sqrt{v})^N}{N!} \left[\sum_{s=1}^n u_s\cos\left(\phi + \frac{2s\pi}{n}\right)\right]^N. 
\end{multline*}
\end{pro}

The series occurring in the integrand is a Fox-Wright confluent hypergeometric function ${}_1\Psi_1$ (\cite{Man-Sri}): 
\begin{equation*}
\sum_{N \geq 0} \frac{\Gamma((N+n)/2) )}{(nk)_N} \frac{z^n}{N!} = \Gamma(nk) {}_1\Psi_1 \left[((n/2), (1/2)), (nk, 1); z\right], \quad z \in \mathbb{C}.
\end{equation*}
Moreover, one obtains a similar integral representation for the density of $V_0$ along the same lines written in the proof of Corollary \ref{Cor1}: 
\begin{cor}
Let $n$ be an odd integer and $k \in (1/2,1]$. Then, the density of $V_0$ admits the following integral representation (up to a constant depending only on $n,k$):
\begin{multline*}
\sin^{2\nu}(n\phi) e^{-v}v^{n\nu-1} \int_{\Sigma_n} du_1\ldots du_{n-1} \prod_{s=1}^n u_s^{k-1}\left\{{}_1\Psi_1 \left[((n/2), (1/2)), (nk, 1); 2\sqrt{v} \sum_{s=1}^n u_s\cos\left(\phi + \frac{2s\pi}{n}\right)\right]\right.
 \\ \left. + {}_1\Psi_1 \left[((n/2), (1/2)), (nk, 1); 2\sqrt{v} \sum_{s=1}^n u_s\cos\left(\phi - \frac{(2s-1)\pi}{n}\right)\right]\right\}.
\end{multline*}
\end{cor}
As to the analogue of Corollary \ref{Cor2}, we similarly notice that 
\begin{equation*}
\sum_{s=1}^n u_s\cos\left(\frac{\pi}{2n}+ \frac{2s\pi}{n}\right) = \sum_{s=1}^n u_s\cos\left(\frac{\pi}{2n} - \frac{(2s-1)\pi}{n}\right)
\end{equation*}
and integrate the density specialized to $\phi = \pi/(2n)$ with respect to the Gamma weight $v^{(n+1)/2 - n\nu} e^{-yv}$ for a given $y > 0$. Using the Gauss duplication formula, the resulting weighted Laplace transform is then given by: 
\begin{cor}
Let $n \geq 3$ be an odd integer and $k > 0$. If the radial Dunkl process associated with the multiplicity value $(1-k)$ starts at $\rho e^{i\pi/(2n)}$, then 
\begin{multline*}
\mathbb{E}_{\rho, \pi/(2n)}^{-\nu}\left(V_0^{(n+1)/2 -n\nu} e^{-yV_0} \right) \propto \frac{1}{(1+y)^{(n+1)/2}} \int_{\Sigma_n} du_1\ldots du_{n-1} \prod_{s=1}^n u_s^{k-1} \\ {}_1F_1 \left[n, nk; \frac{1}{\sqrt{1+y}}\sum_{s=1}^n u_s\cos\left(\frac{(4s+1)\pi}{2n}\right)\right]
\end{multline*}
where the coefficient of proportionality depends only on $n,k$. 
\end{cor}

\end{document}